\documentclass{article}
\usepackage[utf8]{inputenc}

\title{Free metabelian groups are permutation stable}
\author{Hiroki Ishikura}
\date{January 18, 2023}

\usepackage{amsmath}
\usepackage{amssymb}
\usepackage{amsthm}
\usepackage{enumitem}
\usepackage{url}
\usepackage{hyperref}

\numberwithin{equation}{section}

\theoremstyle{plain}
\newtheorem{thm}{Theorem}[section]
\newtheorem{prop}[thm]{Proposition}

\newtheorem{lem}[thm]{Lemma}
\newtheorem{cor}[thm]{Corollary}

\theoremstyle{definition}
\newtheorem{dfn}[thm]{Definition}

\theoremstyle{remark}
\newtheorem{rmk}[thm]{Remark}

\theoremstyle{definition}
\newtheorem{nta}[thm]{Notation}

\begin{document}

\maketitle

\begin{abstract}
We prove that all finitely generated free metabelian groups are permutation stable. This partially answers to the question asked by Levit and Lubotzky whether all finitely generated metabelian groups are permutation stable. Our proof extends the range of application of Levit-Lubotzky's method, which is used to show permutation stability of permutational wreath products of finitely generated abelian groups in \cite{LL22}, to non-split and non-permutational metabelian groups.
\end{abstract}

\section{Introduction}

Permutation stability of groups gets attention in recent years, and many groups have been found to be permutation stable. The aim of this paper is to give a new example of such groups.

\begin{dfn}
Let Sym$(n)$ be the symmetric group of degree $n$. The \textit{Hamming metric} $d_\text{H}$ on Sym$(n)$ is defined by $d_\text{H}(\sigma,\tau)=\frac{1}{n}|\{k\in\{ 1,...,n\}\mid \sigma(k)\neq\tau(k)\}|$.
\end{dfn}

\begin{dfn}
Let $G$ be a countable group. An \textit{almost homomorphism} is a sequence $(\phi_n:G\rightarrow$ Sym$(n))_n$ of maps such that $d_\text{H}(\phi_n(gh),\phi_n(g)\phi_n(h))\rightarrow 0\ (n\rightarrow\infty)$ for every $g,h\in G$.\par
We call $G$ \textit{permutation stable} (or \textit{P-stable} for short) if for every almost homomorphism $(\phi_n:G\rightarrow$ Sym$(n))_n$, there exists a sequence of homomorphisms $(\psi_n:G\rightarrow$ Sym$(n))_n$ such that $d_\text{H}(\phi_n(g),\psi_n(g))\rightarrow 0\ (n\rightarrow\infty)$ for every $g\in G$.
\end{dfn}

Let $F$ be a free group and let $F''$ be its second derived subgroup. The group $F/F''$ is called a \textit{free metabelian group}. Our main theorem is the following:

\begin{thm}\label{intr1}
All finitely generated free metabelian groups are P-stable.
\end{thm}

Moreover, the assumption of finite generation can be removed.

\begin{cor}\label{intr4}
All countable free metabelian groups are P-stable.
\end{cor}

Historically, it is a classical question in operator theory whether ``almost commuting" matrices are ``near" commuting matrices. For example, let $U(n)$ be the unitary group of degree $n$ and $\Vert\cdot\Vert$ a norm on the set of $n\times n$ matrices. If $A_n,B_n\in U(n)$ with $\Vert A_nB_n-B_nA_n\Vert\rightarrow 0\ (n\rightarrow\infty)$, then do there exist $A_n',B_n'\in U(n)$ such that $A_n'B_n'=B_n'A_n'$ and $\Vert A_n-A_n'\Vert+\Vert B_n-B_n'\Vert\rightarrow 0\ (n\rightarrow\infty)$? This asks the stability of the relation $xy=yx$ with respect to the unitary matrices and the norm. The answer to this question is ``No" for the operator norm \cite{V83}, but ``Yes" for the normalized Hilbert-Schmidt norm \cite{G10}. P-stability is a discrete version of this stability, that is, it considers permutations instead of matrices. Indeed if $(U(n),\Vert\cdot\Vert)$ is replaced by $(\text{Sym}(n),d_\text{H})$, then the question asks whether the group $\mathbb{Z}^2=\langle x,y\mid xy=yx\rangle$ is P-stable.\par

The first examples of P-stable groups are found in \cite{GR09}, where all finite groups are shown to be P-stable. In \cite{AP15}, all finitely generated abelian groups (e.g., $\mathbb{Z}^2$) are shown to be P-stable.\par

Recently, great progress was made by \cite{BLT19}, which relates P-stability with invariant random subgroups.

\begin{thm}[{\cite[Theorem 1.3]{BLT19}}]\label{intr2}
Let $G$ be a finitely generated amenable group. Then $G$ is P-stable if and only if every invariant random subgroup of $G$ is co-sofic.
\end{thm}

We refer to Section \ref{prels1} for a definition of the terminology. Invariant random subgroups are a notion generalizing normal subgroups and lattices simultaneously, introduced by \cite{AGV14}. Co-soficity is introduced by \cite{G18}. Theorem \ref{intr2} is applied to finding a lot of P-stable groups, such as virtually polycyclic groups, the Baumslag-Solitar groups $B(1,n)$ for all $n\in \mathbb{Z}$ \cite{BLT19}, the Grigorchuk group \cite{Z19}, and other uncountably many groups \cite{LL19}.\par

However characterizing co-soficity of arbitrary invariant random subgroups is still difficult even for solvable groups. Levit and Lubotzky \cite{LL22} give the following P-stable examples of finitely generated metabelian groups by combining Theorem \ref{intr2} with the pointwise ergodic theorem:

\begin{thm}[{\cite[Theorem 1.4]{LL22}}]\label{intr5}
Let $A$ and $Q$ be finitely generated abelian groups, and $X$ a set on which $Q$ acts with only finitely many orbits. Then the semidirect product $Q\ltimes \bigoplus_X A$ induced from the action of $Q$ on $X$ is P-stable.
\end{thm}

In \cite[Section 1.5]{LL22}, it is asked whether all finitely generated metabelian groups are P-stable. Our Theorem \ref{intr1} answers to this question affirmatively for all free metabelian groups that are basic examples of non-split metabelian groups. 

\begin{rmk}
Let $A,Q,X$ be as in Theorem \ref{intr5}. We regard $N=\bigoplus_X A$ as a $\mathbb{Z}[Q]$-module. An extension of $Q$ by $N$ such that the conjugation on $N$ induces the given $\mathbb{Z}[Q]$-structure is called a \textit{permutational metabelian group} in \cite{LL22}. They conjecture that all permutational metabelian groups are P-stable. The group $F_2/F_2''$ is permutational metabelian, but $F_d/F_d''$ is not if $d\geq 3$ (see Theorem \ref{comm1}).
\end{rmk}

\begin{rmk}
Finitely generated free metabelian groups are already shown to be \textit{Hilbert-Schmidt stable}, that is, they are stable with respect to the unitary groups and the normalized Hilbert-Schmidt norm \cite{LV22}.
\end{rmk}

Our proof of Theorem \ref{intr1} basically follows the method of \cite{LL22}. The main part is the construction of a F{\o}lner sequence with special properties (see Section \ref{prels1}). However there are two differences between the setting of \cite{LL22} and ours. First free metabelian groups are not split metabelian, that is, there is no natural lift of the quotient group $Q=F_d/F_d'$ to $F_d/F_d''$. Hence we have to choose a lift and work with it. Lemma \ref{main1} is an elementary but useful observation. The second difference is that if $d\geq 3$, then $F_d/F_d''$ is not permutational metabelian, and this makes the $\mathbb{Z}[Q]$-structure on the commutator subgroup $N=F_d'/F_d''$ complicated. To apply the method of \cite{LL22}, we need to find a sequence of subgroups $M_n\leq N$ indicated in Theorem \ref{pre5}, and this is the most non-trivial part. Proposition \ref{comm12} is important to find such subgroups.\par

Corollary \ref{intr4} follows from Theorem \ref{intr1} by the universality of free metabelian groups. We note that increasing unions of P-stable groups may not be P-stable in general. For example, there exists an infinitely generated abelian group which is not residually finite. It is an increasing union of finitely generated abelian groups, which are P-stable, but it is not P-stable since amenable P-stable groups must be residually finite \cite{GR09}. \par

\textbf{Organization of the paper.} In Section \ref{prels}, we review the terminology and tools established by \cite{LL22}. In Section \ref{divs}, we consider the division in the ring of Laurent polynomials with finitely many variables and coefficients in $\mathbb{Z}$. This is preparation for Section \ref{comms}, in which we investigate the commutator subgroup of free metabelian groups. Then we prove Theorem \ref{intr1} in Section \ref{mains} and prove Corollary \ref{intr4} in Section \ref{unis}.

\textbf{Acknowledgements.}
I am grateful to my supervisor Professor Yoshikata Kida for his constant support through my master course.

\section{Levit-Lubotzky's method}\label{prels}

\subsection{Weiss approximation}\label{prels1}

\begin{dfn}[{\cite[Definition 15]{G18}}]
Let $G$ be a countable group. Let Sub$(G)$ be the set of subgroups of $G$, and regard it as a closed subspace of $2^G$ with respect to the product topology. We endow the set C(Sub$(G))^*$ with the weak* topology. An \textit{invariant random subgroup} (or \textit{IRS} for short) of $G$ is a Borel probability measure on Sub$(G)$ invariant under the conjugation by $G$. A \textit{finite index} IRS of $G$ is an IRS of $G$ supported on the set of finite index subgroups of $G$. A \textit{co-sofic} IRS of $G$ is a weak*-limit of finite index IRSs of $G$.
\end{dfn}

\begin{dfn}[{\cite[Section 3]{LL22}}]
Let $G$ be a countable group and $H\leq G$ a subgroup. Let $\delta_H\in$ C(Sub$(G))^*$ denote the Dirac measure on $H$. For a non-empty finite subset $F\subset G$, set $F*H=\frac{1}{|F|}\sum_{g\in F}\delta_{gHg^{-1}}\in$ C(Sub$(G))^*$. Let N$_G(H)$ denote the normalizer of $H$ in $G$.\par
A subset $F\subset G$ is a \textit{transversal} for $H$ in $G$ if $G=\bigsqcup_{g\in F}gH$. Also $F$ is a \textit{finite-to-one transversal} for $H$ in $G$ if $F$ is a disjoint union of finitely many transversals for $H$ in $G$.\par
Let $F_n\subset G$ be a sequence of non-empty finite subsets and $K_n\leq G$ a sequence of finite index subgroups. The sequence of pairs $(K_n,F_n)$ is a \textit{Weiss approximation} of $H$ if $F_n$ is a finite-to-one transversal for N$_G(K_n)$ in $G$ for every $n$ and $F_n*H-F_n*K_n\rightarrow 0$ in C(Sub$(G))^*$.
\end{dfn}

\begin{thm}[{\cite[Theorem 3.10]{LL22}}]\label{pre7}
Let $\mu$ be an IRS of an amenable group $G$ and $(F_n)$ a F{\o}lner sequence of $G$. Suppose for $\mu$-a.e.$\ H\in\textup{Sub}(G)$, there exists a sequence of finite index subgroups $K_n\leq G$ such that the sequence of pairs $(K_n,F_n)$ is a Weiss approximation of $H$. Then $\mu$ is a co-sofic IRS of $G$.
\end{thm}

\begin{prop}[{\cite[Proposition 5.3]{LL22}}]\label{pre8}
Let $G$ be a finitely generated group. Suppose a normal subgroup $N\trianglelefteq G$ and the quotient group $G/N$ are abelian. Let $\mu$ be an IRS of $G$. Then for $\mu$-a.e.$\ H\in\textup{Sub}(G)$, $[N:\textup{N}_N(H)]<\infty$ holds.
\end{prop}

The following Corollary is used in \cite[Section 11]{LL22} to prove their main theorem. We state it explicitly to make our argument simple, and give its proof for the reader's convenience.

\begin{cor}[{\cite{LL22}}]\label{pre9}
Let $G$ be a finitely generated group. Suppose $N\trianglelefteq G$ and $Q=G/N$ are abelian groups. Also suppose that every subgroup $R\leq Q$ admits a F{\o}lner sequence $(F_n)$ of $G$ satisfying the following property ($\#_R$):\\
($\#_R$) If $H\leq G$ is a subgroup with $HN/N=R$ and $[N:\textup{N}_N(H)]<\infty$, then there exists a sequence of finite index subgroups $K_n\leq G$ such that the sequence of pairs $(K_n,F_n)$ is a Weiss approximation of $H$.\par
Then $G$ is P-stable.
\end{cor}

\begin{proof}
By Theorem \ref{intr2}, it suffices to show that every IRS of $G$ is co-sofic. Let $\mu$ be an IRS of $G$. By the Krein-Milman theorem, $\mu$ is a weak*-limit of convex combinations of ergodic IRSs of $G$. Thus we may assume that $\mu$ is ergodic. Since the map $H\in$ Sub$(G)\mapsto HN/N\in$ Sub$(Q)$ is $G$-invariant, it is constant on a $\mu$-conull set. Let $R\in$ Sub$(Q)$ be the constant. By Proposition \ref{pre8}, the set
\begin{equation*}
    A:=\{H\in\text{Sub}(G)\mid HN/N=R\text{ and }[N:\textup{N}_N(H)]<\infty\}
\end{equation*}
is $\mu$-conull. By property ($\#_R$), for every $H\in A$, there exists a sequence of finite index subgroups $K_n\leq G$ such that $(K_n,F_n)$ is a Weiss approximation of $H$. By Theorem \ref{pre7}, $\mu$ is co-sofic.
\end{proof}

\subsection{A sufficient condition to be a Weiss approximation}

Let $G$ be a group and $N\trianglelefteq G$ a normal subgroup. Set $Q=G/N$.

\begin{nta}[{\cite[Section 4.1]{LL22}}]\label{pre12}
For a subgroup $H\leq G$, we set
\begin{itemize}
    \item $Q_H=HN/N\leq Q$,
    \item $N_H=N\cap H$, and
    \item let $\alpha_H: Q_H\rightarrow H/N_H$ be the natural isomorphism.
\end{itemize}
 We denote $[H]=[Q_H,N_H,\alpha_H]$. Note that $\overline{\alpha_H(q)}=q$ for every $q\in Q_H$, where $\xi\in H/N_H\mapsto \overline{\xi}\in G/N$ denotes the natural quotient.
\end{nta}

\begin{prop}[{\cite[Proposition 4.1]{LL22}}]\label{pre1}
Let $R\leq Q,\ M\leq N$ be subgroups and $\alpha:R\rightarrow\textup{N}_G(M)/M$ a homomorphism satisfying $\overline{\alpha(q)}=q$ for every $q\in R$. Then there exists a unique subgroup $H\leq G$ such that $[H]=[R,M,\alpha]$.
\end{prop}

Let $q\in Q\mapsto\widehat{q}\in G$ be a (set-theoretic) section of the quotient map. For a subset $I\subset Q$, we set $\widehat{I}=\{\widehat{q}\in G\mid q\in I\}$.

\begin{prop}[{\cite[Proposition 4.3]{LL22}}]\label{pre2}
Let $H\leq G$ be a subgroup and let $I\subset Q$ and $T\subset N$ be transversals for $Q_H$ in $Q$ and $N_H$ in $N$, respectively. Then $\widehat{I}T$ is a transversal for $H$ in $G$.
\end{prop}

Assume that $G$ is finitely generated, and $Q$ and $N$ are abelian in the rest of this section. We use additive notation for the group operation in $N$ (and use multiplicative notation for $Q$). The action of $Q$ on $N$ is induced from the conjugation by $G$. 

\begin{prop}[{\cite[Proposition 5.2]{LL22}}]\label{pre3}
Let $H\leq G$ be a subgroup such that $HN$ is finitely generated. Then for any finite subset $T\subset N$, there exists a finite index subgroup $M\leq N$ such that $N_H\cap T=M\cap T,\ N_H\leq M$, and $M$ is $Q_H$-invariant.
\end{prop}

\begin{dfn}
Let $\Gamma$ and $\Lambda$ be groups and $\Delta,\Delta_n\leq \Lambda$ subgroups for $n\in\mathbb{N}$. Following \cite[Definition 6.1]{LL22}, we say that a sequence of homomorphisms $\phi_n:\Gamma\rightarrow \textup{N}_\Lambda(\Delta_n)/\Delta_n$ is \textit{consistent} with a homomorphism $\phi:\Gamma\rightarrow \textup{N}_\Lambda(\Delta)/\Delta$ if for every $\gamma\in\Gamma$, there exists $\lambda\in\Lambda$ such that $\phi_n(\gamma)=\lambda\Delta_n$ and $\phi(\gamma)=\lambda\Delta$ for every $n$.
\end{dfn}

\begin{prop}[{\cite[Corollary 6.3]{LL22}}]\label{pre4}
Let $H\leq G$ be a subgroup and $N_n\leq N$ be a sequence of $Q_H$-invariant subgroups such that $N_n\rightarrow N_H$ in \textup{Sub}$(N)$. Then there exist homomorphisms $\alpha_n:Q_H\rightarrow \textup{N}_G(N_n)/N_n$ defined for all $n$ sufficiently large such that the sequence $\alpha_n$ is consistent with $\alpha_H$.
\end{prop}

\begin{dfn}[{\cite[Definition 7.2]{LL22}}]
Let $T_n\subset N$ and $F_n\subset G$ be sequences of non-empty finite subsets. The sequence $T_n$ is \textit{adapted} to the sequence $F_n$ if for every $g\in G$ and every finite subset $\Phi\subset N$,
\begin{equation}
    \frac{|\{h\in F_n\mid [g,h]+\Phi\subset T_n\}|}{|F_n|}\rightarrow 1\ (n\rightarrow\infty).\label{pre10}
\end{equation}
\end{dfn}

\begin{rmk}\label{pre11}
By taking $g$ to be the identity of $G$, convergence (\ref{pre10}) implies $\Phi\subset T_n$ for every $n$ sufficiently large. Hence if $T_n$ is adapted to some $F_n$, then $\bigcup_n\bigcap_{k\geq n}T_k=N$ holds.
\end{rmk}

\begin{thm}[{\cite[Theorem 7.5]{LL22}}]\label{pre5}
Let $H\leq G$ be a subgroup with $[N:\textup{N}_N(H)]<\infty$. For every $n\in\mathbb{N}$, let $K_n\leq G$ be a finite index subgroup with $[K_n]=[Q_n, N_n, \alpha_n]$, $M_n\leq N$ a subgroup, and $I_n\subset Q$ and $T_n\subset M_n$ 
non-empty finite subsets. Assume that they satisfy the following conditions:
\begin{enumerate}
    \item $Q_H\leq Q_n$ and $Q_n\rightarrow Q_H$ in \textup{Sub}$(Q)$.
    \item $M_n$ is $Q_H$-invariant.
    \item The sequence $\alpha_n|_{Q_H}$ is consistent with $\alpha_H$.
    \item $N_H\cap T_n=N_n\cap T_n$ and $N_H\cap M_n\leq N_n$.
    \item The sequence $T_n$ is adapted to the sequence $\widehat{I_n}$.
\end{enumerate}
Then for any sequence $P_n\subset M_n$ of non-empty finite subsets, we have
\begin{equation*}
    (\widehat{I_n}P_n)*H - (\widehat{I_n}P_n)*K_n \rightarrow 0\text{ in }\textup{C(Sub}(G))^*
\end{equation*}
\end{thm}

\begin{rmk}
The original statement of \cite[Theorem 7.5]{LL22} is different from the above. They assume the following two conditions (see \cite[Definitions 7.1 and 7.3]{LL22}): 
\begin{enumerate} [label=(\alph*)]
    \item $(K_n,M_n,T_n)$ is a \textit{controlled approximation} of $H$.
    \item $\widehat{I_n}P_n$ is \textit{adapted} to $(K_n,M_n,T_n)$.
\end{enumerate}
Condition (a) claims conditions (i)-(iv) and $\bigcup_n\bigcap_{k\geq n}T_k=N$ holds. The last condition follows from condition (v) by Remark \ref{pre11}. Condition (b) is stronger than condition (v), but they only use condition (v) for the proof of this theorem.
\end{rmk}

\section{Division of Laurent polynomials}\label{divs}

Throughout the paper, we mean by an interval $I\subset \mathbb{R}$ the intersection $I\cap \mathbb{Z}$ if there is no cause of confusion. Let $\mathbb{N}$ be the set of positive integers. Let $\lceil\cdot\rceil:\mathbb{R}\rightarrow\mathbb{Z}$ be the ceiling function, and $\lfloor\cdot\rfloor:\mathbb{R}\rightarrow\mathbb{Z}$ the floor function.

\begin{dfn}
Let $\phi\in\mathbb{Z}[s^{\pm 1}]$ be a nonzero polynomial. Let $m$ be the smallest degree of $s$ in $\phi$ and let $n$ be the non-negative integer such that $m+n$ is the largest degree of $s$ in $\phi$. We call the integer $n$ the \textit{degree} of $\phi$. We say that $\phi$ is \textit{monic} if each of the coefficients of $s^m$ and $s^{m+n}$ belongs to $\{\pm 1\}$.
\end{dfn}

For $n\in\mathbb{N}$, set

\begin{equation*}
    M(n;s)=\text{span}_\mathbb{Z}\{s^m\mid -\lceil n/2\rceil <m\leq\lfloor n/2 \rfloor\}\subset\mathbb{Z}[s^{\pm 1}].
\end{equation*}

\begin{lem}\label{div1}
Let $\phi\in \mathbb{Z}[s^{\pm 1}]$ be a monic polynomial with degree $n>0$. Then for every $\psi\in\mathbb{Z}[s^{\pm 1}]$, there exists $\theta\in\mathbb{Z}[s^{\pm 1}]$ such that $\psi-\theta\phi\in M(n;s)$.
\end{lem}

\begin{proof}
Let $m$ be the smallest degree of $s$ in $\phi$. Since $M(n;s)$ is closed under linear combination with coefficients in $\mathbb{Z}$, we may assume $\psi=s^k$ with $k\in \mathbb{Z}$. Suppose $k>0$. We may assume that the coefficient of $s^{n+m}$ in $\phi$ is $1$. Since $s^{-m}\phi\in\mathbb{Z}[s]$ is a monic polynomial of $s$ with degree $n$ in the usual sense, there exists $\theta\in\mathbb{Z}[s]$ such that the polynomial
\begin{equation*}
    s^{k+\lceil n/2\rceil-1}- \theta\cdot s^{-m}\phi\in\mathbb{Z}[s]
\end{equation*}
of $s$ has degree less than $n$. Then by multiplying $s^{1-\lceil n/2\rceil}$ to this polynomial, we have
\begin{equation*}
    s^k-s^{-m+1-\lceil n/2\rceil}\theta\cdot\phi\in M(n;s).
\end{equation*}
as desired. Next suppose $k\leq 0$. We may assume that the coefficient of $s^{m}$ in $\phi$ is $1$. Since $s^{-n-m}\phi\in\mathbb{Z}[s^{-1}]$ is a monic polynomial of $s^{-1}$ with degree $n$ in the usual sense, there exists $\theta\in\mathbb{Z}[s^{-1}]$ such that the polynomial
\begin{equation*}
    s^{k-\lfloor n/2 \rfloor}-\theta\cdot s^{-n-m}\phi\in\mathbb{Z}[s^{-1}]
\end{equation*}
of $s^{-1}$ has degree less than $n$. Then by multiplying $s^{\lfloor n/2 \rfloor}$ to this polynomial, we have
\begin{equation*}
    s^k-s^{-n-m+\lfloor n/2 \rfloor}\theta\cdot\phi\in M(n;s)
\end{equation*}
as desired.
\end{proof}

Let $d\in\mathbb{N}$. Now we consider the ring $\mathbb{Z}[s_1^{\pm 1},...,s_d^{\pm 1}]$. For $n_1,...,n_d\in\mathbb{N}\cup\{\infty\}$, set
\begin{align*}
    M(n_1,...,n_d;s_1,...,s_d)
    &=\text{span}_\mathbb{Z}\{s_1^{m_1}...s_d^{m_d}\mid m_i\in(-\lceil n_i/2\rceil,\lfloor n_i/2 \rfloor]\text{ for every }i \}\\
    &\subset \mathbb{Z}[s_1^{\pm 1},...,s_d^{\pm 1}],
\end{align*}
where $(-\lceil \infty/2\rceil,\lfloor \infty/2 \rfloor]$ means $(-\infty,\infty)$.

\begin{prop}\label{div2}
Let $1\leq c\leq d$. For every $1\leq i\leq c$, let $\phi_i\in\mathbb{Z}[s_i^{\pm 1}]$ be a monic polynomial with degree $n_i>0$. Then for every $\psi\in \mathbb{Z}[s_1^{\pm 1},...,s_d^{\pm 1}]$, there exist $\theta_i\in\mathbb{Z}[s_1^{\pm 1},...,s_d^{\pm 1}]$ for $1\leq i\leq c$ and $\lambda\in M(n_1,...,n_c,\infty,...,\infty;s_1,...,s_d)$ such that
\begin{equation*}
    \psi=\sum_{i=1}^c\theta_i\phi_i+\lambda.
\end{equation*}
Moreover such $\lambda\in M(n_1,...,n_c,\infty,...,\infty;s_1,...,s_d)$ is unique, that is, if 
\begin{equation*}
    \psi=\sum_{i=1}^c\theta_i'\phi_i+\lambda'
\end{equation*}
with $\theta_i'\in\mathbb{Z}[s_1^{\pm 1},...,s_d^{\pm 1}]$ and $\lambda'\in M(n_1,...,n_c,\infty,...,\infty;s_1,...,s_d)$, then $\lambda=\lambda'$.
\end{prop}

\begin{proof}
First we prove the existence by the induction on $c$. By the induction hypothesis, there exists $\theta_i\in\mathbb{Z}[s_1^{\pm 1},...,s_d^{\pm 1}]$ such that
\begin{equation*}
    \psi-\sum_{i=1}^{c-1}\theta_i\phi_i\in M(n_1,...,n_{c-1},\infty,...,\infty;s_1,...,s_d).
\end{equation*}
Thus we may assume $\psi\in M(n_1,...,n_{c-1},\infty,...,\infty;s_1,...,s_d)$, and further $\psi=s_1^{m_1}...s_{c-1}^{m_{c-1}}\psi'$ with $m_i\in(-\lceil n_i/2\rceil,\lfloor n_i/2 \rfloor]$ and $\psi'\in\mathbb{Z}[s_c^{\pm 1},...,s_d^{\pm 1}]$ since every element of $M(n_1,...,n_{c-1},\infty,...,\infty;s_1,...,s_d)$ is a linear combination with coefficients in $\mathbb{Z}$ of polynomials of this form. Then $\psi'$ is represented as $\psi'=\sum_{k}s_c^k\psi_k'$
for some $\psi_k'\in\mathbb{Z}[s_{c+1}^{\pm 1},...,s_d^{\pm 1}]$. Since $M(n_1,...,n_c,\infty,...,\infty;s_1,...,s_d)$ is closed under multiplying elements of $\mathbb{Z}[s_{c+1}^{\pm 1},...,s_d^{\pm 1}]$, we may assume $\psi'=s_c^k$ with $k\in\mathbb{Z}$. Then by Lemma \ref{div1}, there exists $\theta\in\mathbb{Z}[s_c^{\pm 1}]$ such that $\psi'-\theta \phi_c\in M(n_c;s_c)$. Then we have
\begin{equation*}
    \psi-s_1^{m_1}\cdots s_{c-1}^{m_{c-1}}\theta\cdot\phi_c=s_1^{m_1}\cdots s_{c-1}^{m_{c-1}}(\psi'-\theta \phi_c)\in M(n_1,...,n_c,\infty,...,\infty;s_1,...,s_d).
\end{equation*}
The existence is proved.\par
Now we prove the uniqueness. It suffices to show that if $\psi=\sum_{i=1}^c\theta_i\phi_i\in M(n_1,...,n_c,\infty,...,\infty;s_1,...,s_d)$ with $\theta_i\in\mathbb{Z}[s_1^{\pm 1},...,s_d^{\pm 1}]$, then $\psi=0$. We call such $\Theta:=(\theta_i)_{i=1}^c$ a representation of $\psi$. Suppose $\psi\neq 0$. Let $\sigma(\Theta)$ be the smallest integer $i\in [1,c]$ such that $\theta_i\neq 0$. We may assume that $\sigma(\Theta)$ is the largest among all the representations of $\psi$. Set $j=\sigma(\Theta)$. Let $\tau(\Theta)\geq 0$ be the difference between the largest and the smallest degrees of $s_j$ in $\theta_j$. We may assume that $\tau(\Theta)$ is the smallest among all the representations of $\psi$ satisfying $\sigma(\Theta)=j$. Let $l$ and $m$ be the smallest degrees of $s_j$ in $\theta_j$ and $\phi_j$, respectively. For $j\leq i\leq c$, set $\theta_i=\sum_ks_j^{k}\mu_{i,k}$ with $\mu_{i,k}\in\mathbb{Z}[s_1^{\pm 1},...,s_{j-1}^{\pm 1},s_{j+1}^{\pm 1},...,s_d^{\pm 1}]$. Note that $\theta_j=\sum_{k=m}^{m+\tau(\Theta)}s_j^k\mu_{j,k}$. Also set $\phi_j=\sum_{k=l}^{l+n_j}a_ks_j^k$ with $a_k\in\mathbb{Z}$.\par 
Toward deducing a contradiction, we first assume $m+l\leq -\lceil n_j/2\rceil$. Then the sum of the terms of $\theta_j\phi_j$ whose degree of $s_j$ is $m+l$ is $a_ls_j^{m+l}\mu_{j,m}$. For $i>j$, the sum of the terms of $\theta_i\phi_i$ whose degree of $s_j$ is $m+l$ is $s_j^{m+l}\mu_{i,m+l}\phi_i$. Since $\psi=\sum_{i=j}^c\theta_i\phi_i\in M(n_1,...,n_c,\infty,...,\infty;s_1,...,s_d)$, the sum of the terms of $\psi$ whose degree of $s_j$ is $m+l$ must be $0$. Thus we have
\begin{equation*}
    a_l\mu_{j,m}+\sum_{i=j+1}^c\mu_{i,m+l}\phi_i=0.
\end{equation*}
Note that $a_l\in\{\pm 1\}$ since $\phi_j$ is monic. Hence we have
\begin{align*}
    \psi&=\psi-(\mu_{j,m}+a_l^{-1}\sum_{i=j+1}^c\mu_{i,m+l}\phi_i)s_j^m\phi_j\\
    &=(\theta_j-s_j^m\mu_{j,m})\phi_j+\sum_{i=j+1}^c (\theta_i-a_l^{-1}s_j^m\mu_{i,m+l}\phi_j)\phi_i.
\end{align*}
This gives a new representation $\Theta'=(\theta_i')_{i=1}^c$ of $\psi$ such that $\theta_i'=0$ if $i<j$, and $\theta_j'=\theta_j-s_j^m\mu_{j,m}=\sum_{k=m+1}^{m+\tau(\Theta)}s_j^{k}\mu_{i,k}$. Since $\sigma(\Theta)=j$ is the largest, we have $\sigma(\Theta')=j$. Then, however, $\tau(\Theta')<\tau(\Theta)$ holds, which contradicts the assumption of $\tau(\Theta)$ being the smallest.\par
We next assume $m+l> -\lceil n_j/2\rceil.$ The argument is analogous. The largest degrees of $s_j$ in $\theta_j$ and $\phi_j$ are $m':=m+\tau(\Theta)$ and $l':=l+n_j$, respectively. Then 
\begin{equation*}
    m'+l'\geq m+l+n_j>-\lceil n_j/2\rceil+n_j=\lfloor n_j/2 \rfloor
\end{equation*}
holds.
Since $\psi\in M(n_1,...,n_c,\infty,...,\infty;s_1,...,s_d)$, the sum of the terms of $\psi$ whose degree of $s_j$ is $m'+l'$ must be $0$. Hence we have
\begin{equation*}
    a_{l'}\mu_{j,m'}+\sum_{i=j+1}^c\mu_{i,m'+l'}\phi_i=0
\end{equation*}
and
\begin{align*}
    \psi=(\theta_j-s_j^{m'}\mu_{j,m'})\phi_j+\sum_{i=j+1}^c (\theta_i-a_{l'}^{-1}s_j^{m'}\mu_{i,m'+l'}\phi_j)\phi_i
\end{align*}
in the same way as above. This gives a new representation $\Theta'=(\theta_i')_{i=1}^c$ of $\psi$ such that $\theta_i'=0$ if $i<j$, and $\theta_j'=\theta_j-s_j^{m'}\mu_{j,m'}=\sum_{k=m}^{m+\tau(\Theta)-1}s_j^{k}\mu_{i,k}$. Again we have $\sigma(\Theta')=j$ and $\tau(\Theta')<\tau(\Theta)$, which is a contradiction.
\end{proof}

\section{The commutator subgroup of free metabelian groups}\label{comms}

Let $G$ be a group and $N\trianglelefteq G$ a normal abelian subgroup. Set $Q=G/N$. 

\begin{nta}\label{comm19}
Let $g\in G\mapsto \overline{g}\in Q$ be the quotient map. We use additive notation for the group operation in $N$. The right action $N\curvearrowleft Q$ is induced from the conjugation by $G$, and it is denoted by $f^{\overline{g}}:=g^{-1}fg$ for $f\in N$ and $g\in G$. We regard $N$ as a right $\mathbb{Z}[Q]$-module with respect to this action, and the multiplication of $f\in N$ by $\phi\in\mathbb{Z}[Q]$ is denoted by $f^\phi$. That is, if $\phi=\sum_{q\in Q}n_q\cdot q\in \mathbb{Z}[Q]$ with $n_q\in\mathbb{Z}$, then $f^\phi=\sum_{q\in Q}n_q f^q$ holds.
\end{nta}

\begin{rmk}
We keep using multiplicative notation for the group operation in $Q$ even when $Q$ is abelian. Then for $q_1,q_2\in Q$, the elements $q_1+q_2$ and $q_1q_2$ of $\mathbb{Z}[Q]$ are distinguished. For $f\in N$, note that $f^{q_1+q_2}=g_1^{-1}fg_1+g_2^{-1}fg_2$ and $f^{q_1q_2}=g_2^{-1}g_1^{-1}fg_1g_2$, where $g_1,g_2\in G$ with $\overline{g_1}=q_1$ and $\overline{g_2}=q_2$.
\end{rmk}

\begin{lem}
For all $g,h,k\in G$, we have
\begin{equation}
    [g,hk]=[g,k]+[g,h]^{\overline{k}}. \label{pre6}
\end{equation}
Also for all $g,h\in G$ and $n\in\mathbb{N}$, we have
\begin{equation}
    [g,h^n]=[g,h]^{1+\overline{h}+\cdots+\overline{h}^{n-1}}.\label{comm15}
\end{equation}
\end{lem}

\begin{proof}
In general
\begin{equation*}
    [g,hk]=g^{-1}k^{-1}h^{-1}ghk=g^{-1}k^{-1}gk\cdot k^{-1}g^{-1}h^{-1}ghk=[g,k]\cdot k^{-1}[g,h]k
\end{equation*}
holds, and we have equation (\ref{pre6}). Then we have
\begin{align*}
    [g,h^n]&=[g,h^{n-1}]+[g,h]^{\overline{h}^{n-1}}\\
    &=[g,h^{n-2}]+[g,h]^{\overline{h}^{n-2}}+[g,h]^{\overline{h}^{n-1}}\\
    &=\cdots\\
    &=[g,h]^{1+\overline{h}+\cdots+\overline{h}^{n-1}}.
\end{align*}
\end{proof}

Let $G=F_d/F_d''$ with $d\geq 2$ and $N=G'=F_d'/F_d''$ in the rest of this section. Let $\{a_1,...,a_d\}$ be a free generator of $G$ and set $s_i=\overline{a_i}\in Q$ for $1\leq i\leq d$. Then $Q$ is the free abelian group over the basis $\Sigma_Q:=\{s_1,...,s_d\}$, and the group ring $\mathbb{Z}[Q]$ is identified with $\mathbb{Z}[s_1^{\pm 1},...,s_d^{\pm 1}]$. Since $N$ is the commutator subgroup of $G$, it is generated by the set 
\begin{equation*}
    X:=\{[a_i,a_j]\mid 1\leq i<j\leq d\}
\end{equation*}
as a $\mathbb{Z}[Q]$-module.

\begin{thm}[{\cite[Theorem 3]{B65}}]\label{comm1}
The $\mathbb{Z}[Q]$-module $N$ is naturally isomorphic to $(\bigoplus_{X}\mathbb{Z}[Q])/\mathfrak{I}$, where $\mathfrak{I}$ is the $\mathbb{Z}[Q]$-submodule of $\bigoplus_{X}\mathbb{Z}[Q]$ generated by the subset
\begin{equation*}
   \{[a_i,a_j]^{1-s_k}+[a_i,a_k]^{s_j-1}+[a_j,a_k]^{1-s_i}\mid 1\leq i<j<k\leq d \}.
\end{equation*}
In particular if $d=2$, then $\mathfrak{I}=0$ and $N$ is isomorphic to $\mathbb{Z}[Q]$.
\end{thm}

\begin{cor} \label{comm14}
For all $1\leq i<j<k\leq d$, we have
\begin{equation*}
    [a_i,a_j]^{1-s_k}+[a_i,a_k]^{s_j-1}+[a_j,a_k]^{1-s_i}=0\quad\text{in } N.
\end{equation*}
\end{cor}

\begin{lem}
For all $1\leq i<j<k\leq d$ and $n\in\mathbb{N}$, we have
\begin{align}
    [a_i,a_j]^{s_k^{n}}=[a_i,a_j]+&[a_i,a_k]^{(s_j-1)(1+s_k+...+s_k^{n-1})}+
    [a_j,a_k]^{(1-s_i)(1+s_k+...+s_k^{n-1})}\label{comm16}
\end{align}
and
\begin{align}
    [a_i,a_j]^{s_k^{-n}}=[a_i,a_j]+&[a_i,a_k]^{(1-s_j)(s_k^{-1}+s_k^{-2}+...+s_k^{-n})}+
    [a_j,a_k]^{(s_i-1)(s_k^{-1}+s_k^{-2}+...+s_k^{-n})}\label{comm17}
\end{align}
in $N$.
\end{lem}

\begin{proof}
By Corollary \ref{comm14}, we have
\begin{align*}
    [a_i,a_j]^{s_k^n-1}&=[a_i,a_j]^{(s_k-1)(1+s_k+...+s_k^{n-1})}\\
    &=[a_i,a_k]^{(s_j-1)(1+s_k+...+s_k^{n-1})}+
    [a_j,a_k]^{(1-s_i)(1+s_k+...+s_k^{n-1})}
\end{align*}
and obtain equation (\ref{comm16}). By multiplying $s_k^{-n}$ to equation (\ref{comm16}), we obtain equation (\ref{comm17}).
\end{proof}

For each $1\leq j\leq d$, let $Q^{(j)}$ be the subgroup of $Q$ generated by $\{s_i\}_{i=1}^j$. 

\begin{prop}\label{comm13}
For every $f\in N$, there exists a unique family $(\phi_{i,j})_{1\leq i<j\leq d}$ of elements of $\mathbb{Z}[Q]$ satisfying
\begin{enumerate}
    \item $\phi_{i,j}\in \mathbb{Z}[Q^{(j)}]$ for all $i$ and $j$, and
    \item $f=\sum_{1\leq i<j\leq d}[a_i,a_j]^{\phi_{i,j}}.$
\end{enumerate}
\end{prop}

\begin{proof}
First, we prove the existence. Set 
\begin{equation*}
    \widetilde{N}=\sum_{1\leq i<j\leq d} [a_i,a_j]^{\mathbb{Z}[Q^{(j)}]}.
\end{equation*}
Clearly $X\subset \widetilde{N}$. Since $N$ is generated by $X$ as a $\mathbb{Z}[Q]$-module, it suffices to show that $\widetilde{N}$ is $Q$-invariant. Let $1\leq i<j\leq d$ and $\phi\in \mathbb{Z}[Q^{(j)}]$. We show $[a_i,a_j]^{\phi s_k^{\pm 1}}\in \widetilde{N}$ for every $1\leq k\leq d$. If $k\leq j$, then $\phi s_k^{\pm1}\in \mathbb{Z}[Q^{(j)}]$ and it is done. Suppose $k>j$. Then by equations (\ref{comm16}) and (\ref{comm17}), we have 
\begin{align*}
    &[a_i,a_j]^{\phi s_k}=[a_i,a_j]^\phi+[a_i,a_k]^{\phi(s_j-1)}+[a_j,a_k]^{\phi(1-s_i)}\quad\text{and}\\
    &[a_i,a_j]^{\phi s_k^{-1}}=[a_i,a_j]^\phi+[a_i,a_k]^{\phi (1-s_j)s_k^{-1}}+[a_j,a_k]^{\phi (s_i-1)s_k^{-1}}.
\end{align*}
By looking at the right-hand side, these are elements of $\Tilde{N}$ as desired.\par
Next we prove the uniqueness. By Theorem \ref{comm1}, it suffices to show that $\bigoplus_{1\leq i<j\leq d}[a_i,a_j]^{\mathbb{Z}[Q^{(j)}]}\cap\mathfrak{I}=0$. Suppose $0\neq f\in\bigoplus_{1\leq i<j\leq d}[a_i,a_j]^{\mathbb{Z}[Q^{(j)}]}\cap\mathfrak{I}$. Since $f\in\mathfrak{I}$, there exists a family $\Phi=(\phi_{i,j,k})_{1\leq i<j<k\leq d}$ of elements of $\mathbb{Z}[Q]$ such that
\begin{equation}
    f=\sum_{1\leq i<j<k\leq d} ([a_i,a_j]^{1-s_k}+[a_i,a_k]^{s_j-1}+[a_j,a_k]^{1-s_i})^{\phi_{i,j,k}}. \label{comm2}
\end{equation}
We call such $\Phi$ a representation of $f$. Now we consider the lexicographic order on $\mathbb{N}^3$, that is, 
\begin{equation*}
    (i,j,k)\leq(i',j',k')\Leftrightarrow i<i'\vee (i=i'\wedge j<j') \vee (i=i'\wedge j=j'\wedge k\leq k').
\end{equation*}
Let $\sigma(\Phi)$ be the smallest $(i,j,k)$ such that $\phi_{i,j,k}\neq0$. We may assume that $\sigma(\Phi)$ is the largest among all the representations of $f$. Set $\sigma(\Phi)=(i_0,j_0,k_0)$. By equation (\ref{comm2}), the $[a_{i_0},a_{j_0}]$-th coordinate of $f$ is
\begin{equation}
    \sum_{l<i_0}(1-s_l)\phi_{l,i_0,j_0}+\sum_{i_0<l<j_0}(s_l-1)\phi_{i_0,l,j_0}+\sum_{l>j_0} (1-s_l)\phi_{i_0,j_0,l}. \label{comm3}
\end{equation}
However by the definition of $\sigma(\Phi)$, $\phi_{l,i_0,j_0}=0$ if $l<i_0$, $\phi_{i_0,l,j_0}=0$ if $i_0<l<j_0$, and $\phi_{i_0,j_0,l}=0$ if $j_0<l<k_0$. Thus polynomial (\ref{comm3}) is equal to 
\begin{equation*}
    \sum_{l\geq k_0} (1-s_l)\phi_{i_0,j_0,l}.
\end{equation*}
Since $f\in\bigoplus_{1\leq i<j\leq d}[a_i,a_j]^{\mathbb{Z}[Q^{(j)}]}$, we have $\sum_{l\geq k_0} (1-s_l)\phi_{i_0,j_0,l}\in \mathbb{Z}[Q^{(j_0)}]$, and this must be $0$, which is verified by substituting $s_l=1$ for $l\geq k_0$ (note that $s_l\notin Q^{(j_0)}$). Hence we have
\begin{equation}
    (1-s_{k_0})\phi_{i_0,j_0,k_0}=-\sum_{l>k_0}(1-s_l)\phi_{i_0,j_0,l}. \label{comm4}
\end{equation}
Apply Proposition \ref{div2} to $\{1-s_l\}_{l>k_0}$ and $\phi_{i_0,j_0,k_0}$, and we have
\begin{equation*}
    \phi_{i_0,j_0,k_0}=\sum_{l>k_0}(1-s_l)\psi_l+\lambda
\end{equation*}
for some $\psi_l\in\mathbb{Z}[Q]$ and $\lambda\in\mathbb{Z}[Q^{(k_0)}]$. Then substitute this into equation (\ref{comm4}) and we have
\begin{equation*}
    (1-s_{k_0})\sum_{l>k_0}(1-s_l)\psi_{l}+(1-s_{k_0})\lambda=-\sum_{l>k_0}(1-s_l)\phi_{i_0,j_0,l}.
\end{equation*}
Since $(1-s_{k_0})\lambda\in\mathbb{Z}[Q^{(k_0)}]$, we have $\lambda=0$ by Proposition \ref{div2} (uniqueness of the remainder of division by $\{1-s_l\}_{l>k_0}$). Hence we have
\begin{equation}
    \phi_{i_0,j_0,k_0}=\sum_{l>k_0}(1-s_l)\psi_l. \label{comm5}
\end{equation}
Now we set $\alpha_{i,j,k}=[a_{i},a_{j}]^{1-s_{k}}+[a_{i},a_{k}]^{s_{j}-1}+[a_{j},a_{k}]^{1-s_{i}}$ for $1\leq i<j<k\leq d$. Note that
\begin{equation*}
    \alpha_{i,j,k}^{\quad1-s_{l}}+\alpha_{i,j,l}^{\quad s_{k}-1}+\alpha_{i,k,l}^{\quad 1-s_{j}}+\alpha_{j,k,l}^{\quad s_{i}-1}=0
\end{equation*}
for any $1\leq i<j<k<l\leq d$. Then by equation (\ref{comm5}), we have
\begin{equation*}
    \alpha_{i_0,j_0,k_0}^{\quad\phi_{i_0,j_0,k_0}} =\sum_{l>k_0}\alpha_{i_0,j_0,k_0}^{\quad(1-s_l)\psi_l}
    =\sum_{l>k_0}(\alpha_{i_0,j_0,l}^{\quad1-s_{k_0}}+\alpha_{i_0,k_0,l}^{\quad s_{j_0}-1}+\alpha_{j_0,k_0,l}^{\quad1-s_{i_0}})^{\psi_l}.
\end{equation*}
By adding $0=-\alpha_{i_0,j_0,k_0}^{\quad\phi_{i_0,j_0,k_0}}+\sum_{l>k_0}(\alpha_{i_0,j_0,l}^{\quad1-s_{k_0}}+\alpha_{i_0,k_0,l}^{\quad s_{j_0}-1}+\alpha_{j_0,k_0,l}^{\quad1-s_{i_0}})^{\psi_l}$ to equation (\ref{comm2}), we have
\begin{equation*}
    f=\sum_{\substack{1\leq i<j<k\leq d\\(i,j,k)\neq(i_0,j_0,k_0)}}\alpha_{i,j,k}^{\quad\phi_{i,j,k}}+\sum_{l>k_0}(\alpha_{i_0,j_0,l}^{\quad1-s_{k_0}}+\alpha_{i_0,k_0,l}^{\quad s_{j_0}-1}+\alpha_{j_0,k_0,l}^{\quad1-s_{i_0}})^{\psi_l}.
\end{equation*}
This gives a new representation $\Phi'=(\phi_{i,j,k}')$ of $f$ such that $\phi_{i,j,k}'=\phi_{i,j,k}=0$ if $(i,j,k)<(i_0,j_0,k_0)$, and $\phi_{i_0,j_0,k_0}'=0$. Then we have $\sigma(\Phi')>(i_0,j_0,k_0)=\sigma(\Phi)$, which contradicts the assumption of $\sigma(\Phi)$ being the largest.
\end{proof}

Let $1\leq c\leq d$. Set $U=Q^{(c)}$ and let $V$ be the subgroup of $Q$ generated by $\Sigma_V:=\{s_i\}_{i=c+1}^{d}$.
Fix $n\in\mathbb{N}$. Let $O$ be the $\mathbb{Z}[Q]$-submodule of $N$ generated by the set
\begin{equation*}
    \{[g, a_i^{2n}]\mid g\in G,\ c+1\leq i\leq d\}.
\end{equation*}

\begin{lem}\label{comm9}
If $f\in N$ and $c+1\leq i\leq d$, then $f^{1-s_i^{2n}}=-[f,a_i^{2n}]\in O$ holds.
\end{lem}

\begin{proof}
It follows from
\begin{equation*}
    f^{1-s_i^{2n}}=-f^{s_i^{2n}}+f=a_i^{-2n}f^{-1}a_i^{2n}f=[a_i^{2n},f].
\end{equation*}
\end{proof}

\begin{lem}\label{comm10}
For $1\leq i<j\leq d$, let $O_{i,j}$ be the ideal of ${\mathbb{Z}[Q^{(j)}]}$ generated by
\begin{equation*}
    \{1-s_k^{2n}\mid k\in[c+1,j]\backslash\{i,j\}\}\cup\{1+s_k+\cdots+s_k^{2n-1}\mid k\in[c+1,j]\cap\{i,j\}\}.
\end{equation*}
Then we have
\begin{equation}
    O=\sum_{1\leq i<j\leq d}[a_i,a_j]^{O_{i,j}}. \label{comm6}
\end{equation}
\end{lem}

\begin{proof}
Let $\widetilde{O}$ be the right-hand side of equation (\ref{comm6}). Let $1\leq i<j\leq d$ and $\phi\in O_{i,j}$. We show $[a_i,a_j]^\phi\in O$. We may assume that $\phi$ is in the generator of the ideal $O_{i,j}$, that is, $\phi=1-s_k^{2n}$ with $k\in[c+1,j]\backslash\{i,j\}$, or $\phi=1+s_k+\cdots+s_k^{2n-1}$ with $k\in[c+1,j]\cap\{i,j\}$. In the first case, $[a_i,a_j]^{(1-s_k^{2n})}\in O$ by the previous lemma. In the second case, by equation (\ref{comm15}) we have
\begin{align*}
    [a_i,a_j]^{1+s_j+\cdots+s_j^{2n-1}}&=[a_i,a_j^{2n}]\in O\quad \text{if }k=j, \text{ and}\\
    [a_i,a_j]^{1+s_i+\cdots+s_i^{2n-1}}&=-[a_j,a_i^{2n}]\in O\quad \text{if }k=i.
\end{align*}
Hence $O\supset \widetilde{O}$ holds.\par
Now we show the converse. Note that for all $g,h\in G$ and $c+1\leq k\leq d$,
\begin{align*}
    [gh,a_k^{2n}]&=[g,a_k^{2n}]^{\bar{h}}+[h,a_k^{2n}]
\end{align*}
holds by equation (\ref{pre6}), and
\begin{align*}
    [g^{-1},a_k^{2n}]&=-[g,a_k^{2n}]^{\bar{g}^{-1}}
\end{align*}
holds. Since $G$ is generated by $\{a_i\}_{i=1}^d$, $O$ is generated by 
\begin{equation*}
    A:=\{[a_i,a_k^{2n}]=[a_i,a_k]^{1+s_k+\cdots+s_k^{2n-1}}|\ 1\leq i\leq d,\ c+1\leq k\leq d\}
\end{equation*}
as a $\mathbb{Z}[Q]$-module. Let $1\leq i\leq d$ and $c+1\leq k\leq d$. Then we have 
\begin{align*}
    1+s_k+\cdots+s_k^{2n-1}&\in O_{i,k}\quad\text{if }i<k,\text{ and}\\
    1+s_k+\cdots+s_k^{2n-1}&\in O_{k,i}\quad\text{if }i>k.
\end{align*}
If $i=k$, then $[a_i,a_k]=0$. It follows that $A\subset \widetilde{O}$. Thus it suffices to show that $\widetilde{O}$ is $Q$-invariant. Let $1\leq i<j\leq d$ and $\phi\in O_{i,j}$. We show that $[a_i,a_j]^{s_k^{\pm 1}\phi}\in \widetilde{O}$ for every $1\leq k\leq d$. Since $O_{i,j}$ is $Q^{(j)}$-invariant, we may assume $k>j$. Then by equation (\ref{comm16}) and (\ref{comm17}) we have
\begin{align*}
    &[a_i,a_j]^{s_k\phi}=[a_i,a_j]^\phi+[a_i,a_k]^{(s_j-1)\phi}+[a_j,a_k]^{(1-s_i)\phi}\quad\text{and}\\
    &[a_i,a_j]^{s_k^{-1}\phi}=[a_i,a_j]^\phi+[a_i,a_k]^{(1-s_j)s_k^{-1}\phi}+[a_j,a_k]^{(s_i-1)s_k^{-1}\phi}.
\end{align*}
It suffices to show that $(1-s_j)\phi\in O_{i,k}$ and $(1-s_i)\phi\in O_{j,k}$. Since $\phi\in O_{i,j}$, there exist $\psi_l\in\mathbb{Z}[Q^{(j)}]$ for $c+1\leq l\leq d$ such that 
\begin{equation*}
    \phi=\sum_{l\in[c+1,j]\backslash\{i,j\}}(1-s_l^{2n})\psi_l+\sum_{l\in[c+1,j]\cap\{i,j\}}(1+s_l+...+s_l^{2n-1})\psi_l.
\end{equation*}
Then we have
\begin{align*}
    (1-s_j)\phi=&\sum_{l\in[c+1,j]\backslash\{i,j\}}(1-s_j)(1-s_l^{2n})\psi_l\\
    &+1_{[c+1,d]}(i)\cdot(1-s_j)(1+s_i+\cdots+s_i^{2n-1})\psi_i\\
    &+1_{[c+1,d]}(j)\cdot(1-s_j^{2n})\psi_j,
\end{align*}
and thus $(1-s_j)\phi\in O_{i,k}$. In the same way, $(1-s_i)\phi\in O_{j,k}$ holds. Hence $[a_i,a_j]^{s_k^{\pm 1}\phi}\in \widetilde{O}$ holds.
\end{proof}

\begin{dfn}\label{comm8}
For $1\leq i<j\leq d$, define the subgroup $M_{i,j}\leq \mathbb{Z}[Q^{(j)}]$ as follows:
\begin{enumerate}
    \item If $j\leq c$, then let $M_{i,j}=\mathbb{Z}[Q^{(j)}]$.
    \item If $j\geq c+1$, then let $M_{i,j}$ be the $\mathbb{Z}[U]$-submodule of $\mathbb{Z}[Q^{(j)}]$ generated by
    \begin{equation*}
    \left\{ s_{c+1}^{m_{c+1}}\cdots s_j^{m_j}\;\middle|
    \begin{array}{l}
        m_{k}\in[-n+1,n]\text{ for } k\in[c+1,j]\backslash\{i,j\},\text{ and }\\
        m_{k}\in[-n+1,n-1]\text{ for }k\in[c+1,j]\cap\{i,j\}
    \end{array}
    \right\}.
    \end{equation*}
\end{enumerate}
Then we set
\begin{equation}
    M=\sum_{1\leq i<j\leq d}[a_i,a_j]^{M_{i,j}}. \label{comm7}
\end{equation}
We call $M$ the \textit{$U$-residue of $O$}.
\end{dfn}

\begin{rmk}
Let $1\leq i<j\leq d$. For $1\leq k\leq j$, define $n_k\in\mathbb{N}\cup\{\infty\}$ by
\begin{align*}
    n_k&=\infty\quad\text{if }k\in [1,c],\\
    n_k&=2n\quad\text{if }k\in[c+1,j]\backslash\{i,j\},\text{ and}\\
    n_k&=2n-1\quad\text{if }k\in[c+1,j]\cap\{i,j\}.
\end{align*}
Then $M_{i,j}=M(n_1,...,n_j;s_1,...,s_j)$ holds.
\end{rmk}

\begin{nta}\label{comm22}
Let $A$ be a finitely generated free abelian group over the basis $\Sigma_A=\{s_1,...,s_d\}$. For $m\in \mathbb{N}$, we set
\begin{equation*}
    \text{B}_A(m,\Sigma_A)=\{s_1^{k_1}\cdots s_d^{k_d}\in A\mid -\lceil m/2\rceil< k_i\leq \lfloor m/2 \rfloor\text{ for every }i\}.
\end{equation*}
\end{nta}

\begin{lem}\label{comm11}
The subgroup $M$ is $U$-invariant and contains 
\begin{equation*}
    \{x^v\mid x\in X,\ v\in \textup{B}_Q(2n-1,\Sigma_Q)\}.
\end{equation*}
\end{lem}

\begin{proof}
Let $1\leq i<j\leq d$ and $\phi\in M_{i,j}$. We show that $[a_i,a_j]^{s_k^{\pm 1}\phi}\in M$ for every $1\leq k\leq c$. If $j\geq c+1$, then $s_k^{\pm 1}\phi\in M_{i,j}$ since $M_{i,j}$ is $U$-invariant, and it is done. Now suppose $j\leq c$. If $k\leq j\leq c$, then $s_k^{\pm 1}\phi\in \mathbb{Z}[Q^{(j)}]=M_{i,j}$, and it is done. If $j<k\leq c$, then by equations (\ref{comm16}) and (\ref{comm17}) we have
\begin{align*}
    &[a_i,a_j]^{s_k\phi}=[a_i,a_j]^\phi+[a_i,a_k]^{(s_j-1)\phi}+[a_j,a_k]^{(1-s_i)\phi}\quad\text{and}\\
    &[a_i,a_j]^{s_k^{-1}\phi}=[a_i,a_j]^\phi+[a_i,a_k]^{(1-s_j)s_k^{-1}\phi}+[a_j,a_k]^{(s_i-1)s_k^{-1}\phi}.
\end{align*}
Since $\phi,s_i,s_j,s_k\in\mathbb{Z}[Q^{(k)}]=M_{i,k}=M_{j,k}$, we have $[a_i,a_j]^{s_k^{\pm 1}\phi}\in M$. Hence $M$ is $U$-invariant.\par
For the remaining claim, it suffices to show that
\begin{equation*}
    \{x^v\mid x\in X,\ v\in \text{B}_V(2n-1,\Sigma_V)\}\subset M
\end{equation*}
since $M$ is $U$-invariant. Let $1\leq i<j\leq d$ and $v=s_{c+1}^{m_{c+1}}\cdots s_d^{m_d}\in \text{B}_V(\Sigma_V,2n-1)$. We show that $[a_i,a_j]^v\in M$. If $v$ is the identity of $Q$, it is clear. Otherwise let $k\in[c+1,d]$ be the largest integer such that $m_k\neq 0$. We prove by the induction on $k$. If $k\leq j$, then $v\in M_{i,j}$ since $m_{c+1},...,m_{k}\in[-n+1,n-1]$, and thus $[a_i,a_j]^v\in M$ holds. Suppose $k>j$ and set $w=s_k^{-m_k}v=s_{c+1}^{m_{c+1}}\cdots s_{k-1}^{m_{k-1}}$. If $m_k>0$, then by equation (\ref{comm16}) we have
\begin{align*}
    [a_i,a_j]^{v}=[a_i,a_j]^{s_k^{m_k}w}
    =&[a_i,a_j]^w+[a_i,a_k]^{(s_j-1)(1+s_k+\cdots+s_k^{m_k-1})w}\\
    &+[a_j,a_k]^{(1-s_i)(1+s_k+\cdots+s_k^{m_k-1})w}.
\end{align*}
If $m_k<0$, then by equation (\ref{comm17}) we have
\begin{align*}
    [a_i,a_j]^v=[a_i,a_j]^{s_k^{m_k}w}=&[a_i,a_j]^{w}+[a_i,a_k]^{(1-s_j)(s_k^{-1}+s_k^{-2}+\cdots+s_k^{m_k})w}\\
    &+[a_j,a_k]^{(s_i-1)(s_k^{-1}+s_k^{-2}+\cdots+s_k^{m_k})w}.
\end{align*}
By the induction hypothesis, $[a_i,a_j]^w\in M$. It suffices to show that
\begin{equation*}
    (1-s_j)s_k^{m}w\in M_{i,k} \text{ and } (1-s_i)s_k^mw\in M_{j,k}
\end{equation*}
for every $m\in[-n+1,n-1]$. This follows from
\begin{align}
    s_k^mw&=s_{c+1}^{m_{c+1}}\cdots s_{k-1}^{m_{k-1}}s_k^{m}\in M_{i,k}\cap M_{j,k},\label{comm18}\\
    s_js_k^mw&=s_js_{c+1}^{m_{c+1}}\cdots s_{k-1}^{m_{k-1}}s_k^{m}\in M_{i,k}\text{ and}\label{comm20}\\
    s_is_k^mw&=s_is_{c+1}^{m_{c+1}}\cdots s_{k-1}^{m_{k-1}}s_k^{m}\in M_{j,k}.\label{comm21}
\end{align}
Claim (\ref{comm18}) holds since $m_{c+1},...,m_{k-1}\in[-n+1,n-1]$. If $j\leq c$, then $s_js_k^mw\in M_{i,k}$ since $M_{i,k}$ is $U$-invariant. If $j\geq c+1$, then $s_js_k^mw\in M_{i,k}$ since $m_j+1\in[-n+1,n]$. Thus claim (\ref{comm20}) holds. Claim (\ref{comm21}) also holds in the same way. Hence $[a_i,a_j]^v\in M$ is proved.
\end{proof}

\begin{prop}\label{comm12}
Let $\pi:N\rightarrow N/O$ be the quotient map. Then the restriction $\pi|_{M}:M\rightarrow N/O$ is a bijection.
\end{prop}

\begin{proof}
Let $1\leq i<j\leq d$. For $1\leq k\leq j$, we define $\phi_k\in\mathbb{Z}[s_k^{\pm 1}]$ and $n_k\in\mathbb{N}\cup\{\infty\}$ by
\begin{align*}
    &\phi_k=0,\ n_k=\infty\quad\text{if }k\in [1,c],\\
    &\phi_k=1-s_k^{2n},\ n_k=2n\quad\text{if }k\in[c+1,j]\backslash\{i,j\},\text{ and}\\
    &\phi_k=1+s_k+\cdots+s_k^{2n-1},\ n_k=2n-1\quad\text{if }k\in[c+1,j]\cap\{i,j\}.
\end{align*}
Then $O_{i,j}$ is the ideal of $\mathbb{Z}[Q^{(j)}]$ generated by $\{\phi_k\}_{k=1}^j$, and the equation
\begin{equation*}
    M_{i,j}=M(n_1,...,n_j;s_1,...,s_j)
\end{equation*}
holds. By Proposition \ref{div2}, we have $M_{i,j}+O_{i,j}=\mathbb{Z}[Q^{(j)}]$ and $M_{i,j}\cap O_{i,j}=0$.
\par
Now we show the proposition. By Proposition \ref{comm13} we have
\begin{equation*}
    N=\sum_{1\leq i<j\leq d}[a_i,a_j]^{\mathbb{Z}[Q^{(j)}]}.
\end{equation*}
Also by equations (\ref{comm7}) and (\ref{comm6}), we have
\begin{equation*}
    M=\sum_{1\leq i<j\leq d}[a_i,a_j]^{M_{i,j}}\text{ and }
    O=\sum_{1\leq i<j\leq d}[a_i,a_j]^{O_{i,j}}.
\end{equation*}
Hence we have
\begin{equation*}
    M+O=\sum_{1\leq i<j\leq d}[a_i,a_j]^{M_{i,j}+O_{i,j}}=\sum_{1\leq i<j\leq d}[a_i,a_j]^{\mathbb{Z}[Q^{(j)}]}=N
\end{equation*}
and the surjectivity is proved.\par
Again by Proposition \ref{comm13} (uniqueness), we have
\begin{align*}
   M\cap O&= \left(\sum_{1\leq i<j\leq d}[a_i,a_j]^{M_{i,j}}\right)\cap \left(\sum_{1\leq i<j\leq d}[a_i,a_j]^{O_{i,j}}\right)\\
   &=\sum_{1\leq i<j\leq d}[a_i,a_j]^{M_{i,j}\cap O_{i,j}}=0.
\end{align*}
The injectivity is proved.
\end{proof}

\section{Proof of Theorem \ref{intr1}}\label{mains}

Let $G=F_d/F_d''$ for $d\geq 2$. Set $N=G'=F_d'/F_d''$ and $Q=G/N$. We follow Notation \ref{comm19}.\par
Let $R\leq Q$ be a subgroup. By Corollary \ref{pre9}, it suffices to construct a F{\o}lner sequence $(F_n)$ of $G$ satisfying property ($\#_R$) in Corollary \ref{pre9}.

\subsection{Construction of the F{\o}lner sequence}\label{mains1}

Take subgroups $U,V\leq Q$ such that $Q=U\times V$ and $R$ is a finite index subgroup of $U$. Set $c=\text{rank}(U)$. Let $\{s_1,...,s_c\}$ and $\{s_{c+1},...,s_{d}\}$ be generators of $U$ and $V$, respectively. We may assume that $\overline{a_i}=s_i$ for every $1\leq i\leq d$ by taking a suitable free generator $\{a_1,...,a_d\}$ of $G$ since the natural map Aut$(F_d)\rightarrow$ Aut$(F_d/F_d')$ is surjective \cite[Chapter I, Proposition 4.4]{LS01}. Set $\Sigma_Q=\{s_1,...,s_d\}$.\par
Define the section $q\in Q\mapsto\widehat{q}\in G$ of the quotient map by $\widehat{q}=a_1^{k_1}\cdots a_d^{k_d}$ if $q=s_1^{k_1}\cdots s_d^{k_d}$ with $k_1,...,k_d\in\mathbb{Z}$.

\begin{lem}\label{main1}
Let $g\in G$ and $q=s_1^{\epsilon_1 k_1}\cdots s_d^{\epsilon_d k_d}\in Q$ with $\epsilon_i\in\{\pm 1\}$ and $k_i\geq 0$ for $1\leq i\leq d$. Then we have
\begin{equation*}
    [g,\widehat{q}]=\sum_{\substack{1\leq i\leq d\\:k_i\geq 1}}\sum_{l=0}^{k_i-1} [g,a_i^{\epsilon_i}]^{s_i^{\epsilon_il}s_{i+1}^{\epsilon_{i+1}k_{i+1}}\cdots s_d^{\epsilon_dk_d}}.
\end{equation*}
\end{lem}

\begin{proof}
By equation (\ref{pre6}), we have
\begin{align*}
     [g,\widehat{q}]&=[g,a_1^{\epsilon_1k_1}\cdots a_d^{\epsilon_dk_d}]\\
     &=[g,a_1^{\epsilon_1k_1}]^{s_2^{\epsilon_2k_2}\cdots s_d^{\epsilon_dk_d}}+[g,a_2^{\epsilon_2k_2}\cdots a_d^{\epsilon_dk_d}]\\
     &=\cdots \\
     &=\sum_{i=1}^d[g,a_i^{\epsilon_ik_i}]^{s_{i+1}^{\epsilon_{i+1}k_{i+1}}\cdots s_d^{\epsilon_dk_d}}.
\end{align*}
Then for every $i$ with $k_i\geq 1$, we have
\begin{align*}
    [g,a_i^{\epsilon_ik_i}]=\sum_{l=0}^{k_i-1} [g,a_i^{\epsilon_i}]^{s_i^{\epsilon_il}}
\end{align*}
by equation (\ref{comm15}). Hence the lemma holds.
\end{proof}

In this subsection, we construct subgroups $Q_n\leq Q$ and $M_n\leq N$, and finite subsets $I_n\subset Q$ and $T_n,P_n\subset M_n$ such that $F_n=\widehat{I_n}P_n$ is a F{\o}lner sequence of $G$. We will apply Theorem \ref{pre5} to them in the next subsection.\par

For $n\in\mathbb{N}$, we set $m_n=n\cdot [U:R]\in\mathbb{N}$, $V_n=\{v^{2m_n}\mid v\in V\}\leq V$ and $Q_n=R\times V_n\leq Q$. Recall Notation \ref{comm22}. Then the following is clear.

\begin{lem}\label{main5}
\textup{}
\begin{enumerate}
    \item We have $R\leq Q_n$ and $Q_n\rightarrow R$ in \textup{Sub}$(Q)$.
    \item The subset $I_n:=\textup{B}_Q(2m_n,\Sigma_Q)$ is a finite-to-one transversal for $Q_n$ in $Q$.
\end{enumerate}
\end{lem}

Let $O_n$ be the $\mathbb{Z}[Q]$-submodule of $N$ generated by 
\begin{equation*}
    \{[g, a_i^{2m_n}]\mid g\in G,\ c+1\leq j\leq d\}.
\end{equation*}
We regard $N/O_n$ as a $\mathbb{Z}[Q]$-module. Let $\pi_n:N\rightarrow N/O_n$ be the quotient map.
Let $M_n\leq N$ be the $U$-residue of $O_n$ (Definition \ref{comm8}). The following is a consequence of our observation in Section \ref{comms}.

\begin{lem}\label{main2}
\text{   }
\begin{enumerate}
    \item The action of $V_n$ on $N/O_n$ is trivial.
    \item The subgroup $M_n$ is $R$-invariant.
    \item The map $\pi_n|_{M_n}:M_n\rightarrow N/O_n$ is a bijection.
\end{enumerate}
\end{lem}

\begin{proof}
(i) Since $V_n$ is generated by $\{s_i^{2m_n}\}_{i=c+1}^d$, it follows from Lemma \ref{comm9}.\par
(ii) Since $R\leq U$ and $M_n$ is $U$-invariant by Lemma \ref{comm11}, the claim holds.\par
(iii) This follows from Proposition \ref{comm12}.
\end{proof}

\begin{lem}\label{main14}
Set $\pi_n'=(\pi_n|_{M_n})^{-1}\circ\pi_n:N\rightarrow M_n$. If $P$ is a finite-to-one transversal for a subgroup $K\leq M_n$ in $M_n$, then $P$ is also a finite-to-one transversal for $\pi_n'^{-1}(K)$ in $N$. 
\end{lem}

\begin{proof}
Let $P$ be a transversal for a subgroup $K\leq M_n$ in $M_n$. Then $M_n=\bigsqcup_{p\in P}pK$ holds. Since $\pi_n'|_{M_n}=$ id$_{M_n}$, we have
\begin{equation*}
    N=\pi_n'^{-1}(M_n)=\pi_n'^{-1}(\bigsqcup_{p\in P}pK)=\bigsqcup_{p\in P}p\cdot\pi_n'^{-1}(K).
\end{equation*}
Hence $P$ is a transversal for $\pi_n'^{-1}(K)$ in $N$.
\end{proof}

We set
\begin{align*}
    X&=\{[a_i,a_j]\in N\mid 1\leq i<j\leq d\},\\
    Z_n&=\{x^q\in N\mid x\in X,\ q\in \text{B}_Q(2m_n-1,\Sigma_Q)\}\text{ and}\\
    T_n&=\left\{\sum_{z\in Z_n}k_zz\in N\;\middle|\;|k_z|\leq n^2\text{ for every }z\in Z_n\right\}.
\end{align*}

Note that $T_n\subset T_{n+1}$ for every $n$ and $\bigcup_nT_n=N$. By Lemma \ref{comm11}, $Z_n\subset M_n$ and thus $T_n\subset M_n$ holds.

\begin{lem}\label{main12}
The sequence $T_n$ is adapted to the sequence $\widehat{I_n}$.
\end{lem}

\begin{proof}
Let $g\in G$ and let $\Phi\subset N$ be a finite subset. Take $n_0\in\mathbb{N}$ so that 
\begin{equation*}
\{[g,a_i^{\pm 1}]\mid 1\leq i\leq d\}\cup\Phi\subset T_{n_0}.
\end{equation*}
Set $J_n=\text{B}_Q(2(m_n-m_{n_0})+1,\Sigma_Q)\subset I_n$ for $n>n_0$. Note that $Z_{n_0}^q\subset Z_n$ for every $q\in J_n$ since $J_n\text{B}_Q(2m_{n_0}-1,\Sigma_Q)\subset \text{B}_Q(2m_n-1,\Sigma_Q)$. Thus for every $q\in J_n$, we have
\begin{equation*}
    T_{n_0}^q\subset\left\{\sum_{z\in Z_n}k_zz\in N\;\middle|\;|k_z|\leq n_0^2\text{ for every }z\in Z_n\right\}.
\end{equation*}
Since $n^2/m_n\rightarrow\infty$, we have $n_0^2(m_nd+1)\leq n^2$ for every $n$ sufficiently large. Fix such $n>n_0$. We show that $[g,\widehat{q}]+\Phi\subset T_n$ for every $q\in J_n$. If $q=s_1^{\epsilon_1 k_1}...s_d^{\epsilon_d k_d}\in J_n$ with $\epsilon_i\in\{\pm 1\}$ and $k_i\geq 0$, then by Lemma \ref{main1} we have
\begin{equation}
    [g,\widehat{q}]=\sum_{\substack{1\leq i\leq d\\:k_i\geq 1}}\sum_{l=0}^{k_i-1} [g,a_i^{\epsilon_i}]^{s_i^{\epsilon_il}s_{i+1}^{\epsilon_{i+1}k_{i+1}}...s_d^{\epsilon_dk_d}}.\label{main10}
\end{equation}
Then each term $[g,a_i^{\epsilon_i}]^{s_i^{\epsilon_il}s_{i+1}^{\epsilon_{i+1}k_{i+1}}\cdots s_d^{\epsilon_dk_d}}$ of the right-hand side is in the set
\begin{equation*}
    \left\{\sum_{z\in Z_n}k_zz\in N\;\middle|\;|k_z|\leq n_0^2\text{ for every }z\in Z_n\right\}
\end{equation*}
since $[g,a_i^{\epsilon_i}]\in T_{n_0}\text{ and }s_i^{\epsilon_il}s_{i+1}^{\epsilon_{i+1}k_{i+1}}\cdots s_d^{\epsilon_dk_d}\in J_n.$ Also since the right-hand side of equation (\ref{main10}) has at most $m_nd$ terms,
\begin{equation*}
    [g,\widehat{q}]\in\left\{\sum_{z\in Z_n}k_zz\in N\;\middle|\;|k_z|\leq n_0^2m_nd\text{ for every }z\in Z_n\right\}.
\end{equation*}
It follows from $\Phi\subset T_{n_0}$ and $n_0^2(m_nd+1)\leq n^2$ that
\begin{equation*}
    [g,\widehat{q}]+\Phi\subset\left\{\sum_{z\in Z_n}k_zz\in N\;\middle|\;|k_z|\leq n_0^2(m_nd+1)\text{ for every }z\in Z_n\right\}\subset T_n.
\end{equation*}
Hence we have
\begin{equation*}
    \frac{|\{q\in I_n\mid [g,\widehat{q}]+\Phi\subset T_n\}|}{|I_n|}\geq\frac{|J_n|}{|I_n|}\rightarrow 1\quad(n\rightarrow \infty).
\end{equation*}
\end{proof}

Let $L_n$ be the subgroup of $N$ generated by $Z_n$. Let $Y_n\subset L_n$ be a basis of $L_n$ as a free abelian group. Note that the sequence
\begin{equation*}
    \left\{\sum_{y\in Y_n}k_y y\in L_n\;\middle|\;k_y\in [0,l]\text{ for every }y\in Y_n\right\}\quad(l=1,2,...)
\end{equation*}
is a F{\o}lner sequence of $L_n$. Thus we can take $l_n\in\mathbb{N}$ so that the set
\begin{equation*}
P_n:=\left\{\sum_{y\in Y_n}k_y y\in L_n\;\middle|\;k_y\in [0,l_nn!-1]\text{ for every }y\in Y_n\right\}
\end{equation*}
is $(T_n,1/n)$-invariant, that is, $|P_n\cap(f+P_n)|\geq (1-1/n)|P_n|$ for every $f\in T_n$. Since $Z_n\subset M_n$, we have $P_n\subset L_n\subset M_n$.

\begin{lem}\label{main16}
The sequence $F_n:=\widehat{I_n}P_n$ is a F{\o}lner sequence of $G$.
\end{lem}

\begin{proof}
Since $G$ is generated by $\{a_i\}_{i=1}^d$, it suffices to show that 
\begin{equation*}
     \frac{|F_n\cap a_iF_n|}{|F_n|}\rightarrow 1\quad (n\rightarrow \infty)
\end{equation*}
for every $1\leq i\leq d$. Fix $1\leq i\leq d$. Take $n_0\in\mathbb{N}$ so that $\{[a_i,a_j^{\pm 1}]\}_{j=1}^i\subset T_{n_0}$. Set $J_n=\text{B}_Q(2(m_n-m_{n_0})+1),\Sigma_Q)\subset I_n$ for $n>n_0$. Then for every $q\in J_n$, we have
\begin{equation*}
    T_{n_0}^q\subset\left\{\sum_{z\in Z_n}k_zz\in N\;\middle|\;|k_z|\leq n_0^2\text{ for every }z\in Z_n\right\}
\end{equation*}
as proved in the proof of Lemma \ref{main12}. Since $n^2/m_n\rightarrow\infty$, we have $n_0^2m_ni\leq n^2$ for every $n$ sufficiently large. Fix such $n>n_0$. Let $q=s_1^{\epsilon_1k_1}...s_d^{\epsilon_dk_d}\in Q$ with $\epsilon_j\in\{\pm 1\}$ and $k_j\geq 0$ for $1\leq j\leq d$. Then we have
\begin{align*}
    a_i\widehat{q}&=a_ia_1^{\epsilon_1k_1}\cdots a_d^{\epsilon_dk_d}\notag\\
    &= a_ia_1^{\epsilon_1k_1}\cdots a_i^{\epsilon_ik_i}\cdot a_{i+1}^{\epsilon_{i+1}k_{i+1}}\cdots a_d^{\epsilon_dk_d}\notag\\
    &=a_1^{\epsilon_1k_1}\cdots a_i^{\epsilon_ik_i} a_i[a_i, a_1^{\epsilon_1k_1}\cdots a_i^{\epsilon_ik_i}]\cdot a_{i+1}^{\epsilon_{i+1}k_{i+1}}\cdots a_d^{\epsilon_dk_d}\notag\\ &=a_1^{\epsilon_1k_1}\cdots a_i^{\epsilon_ik_i}a_ia_{i+1}^{\epsilon_{i+1}k_{i+1}}\cdots a_d^{\epsilon_dk_d}\cdot[a_i, a_1^{\epsilon_1k_1}\cdots a_i^{\epsilon_ik_i}]^{s_{i+1}^{\epsilon_{i+1}k_{i+1}}\cdots s_d^{\epsilon_dk_d}},
\end{align*}
where the last equation follows from
\begin{align*}
    (a_{i+1}^{\epsilon_{i+1}k_{i+1}}\cdots a_d^{\epsilon_dk_d})^{-1}[a_i, a_1^{\epsilon_1k_1}\cdots a_i^{\epsilon_ik_i}](a_{i+1}^{\epsilon_{i+1}k_{i+1}}\cdots a_d^{\epsilon_dk_d})\\
    =[a_i, a_1^{\epsilon_1k_1}\cdots a_i^{\epsilon_ik_i}]^{s_{i+1}^{\epsilon_{i+1}k_{i+1}}\cdots s_d^{\epsilon_dk_d}}.
\end{align*}
Set $f_q=[a_i, a_1^{\epsilon_1k_1}\cdots a_i^{\epsilon_ik_i}]^{s_{i+1}^{\epsilon_{i+1}k_{i+1}}\cdots s_d^{\epsilon_dk_d}}$. Then
\begin{equation}
    a_i\widehat{q} =a_1^{\epsilon_1k_1}\cdots a_i^{\epsilon_ik_i}a_ia_{i+1}^{\epsilon_{i+1}k_{i+1}}\cdots a_d^{\epsilon_dk_d}\cdot f_q \label{main15}
\end{equation}
holds. We show that $f_q\in T_n$ if $q\in J_n$.
By Lemma \ref{main1}, we have
\begin{equation}
    f_q=[a_i,a_1^{\epsilon_1k_1}...a_i^{\epsilon_ik_i}]^{s_{i+1}^{\epsilon_{i+1}k_{i+1}}\cdots s_d^{\epsilon_dk_d}}=\sum_{\substack{1\leq j\leq i\\:k_j\geq 1}}\sum_{l=0}^{k_j-1} [a_i,a_j^{\epsilon_j}]^{s_j^{\epsilon_jl}s_{j+1}^{\epsilon_{j+1}k_{j+1}}\cdots s_d^{\epsilon_dk_d}}.\label{main11}
\end{equation}
If $q\in J_n$, then each term $[a_i,a_j^{\epsilon_j}]^{s_j^{\epsilon_jl}s_{j+1}^{\epsilon_{j+1}k_{j+1}}\cdots s_d^{\epsilon_dk_d}}$ of the right-hand side is in 
\begin{equation*}
    \left\{\sum_{z\in Z_n}k_zz\in N\;\middle|\;|k_z|\leq n_0^2\text{ for every }z\in Z_n\right\}
\end{equation*}
since $[a_i,a_j^{\epsilon_j}]\in T_{n_0}$ and $s_j^{\epsilon_jl}s_{j+1}^{\epsilon_{j+1}k_{j+1}}\cdots s_d^{\epsilon_dk_d}\in J_n$. Also since the right-hand side of equation (\ref{main11}) has at most $m_ni$ terms, we have
\begin{align*}
    f_q\in \left\{\sum_{z\in Z_n}k_zz\in N\;\middle|\;|k_z|\leq n_0^2m_ni\text{ for every }z\in Z_n\right\}\subset T_n
\end{align*}
as desired. The last inclusion follows from $n_0^2m_ni\leq n^2$.\par 
Now since $P_n$ is $(T_n,1/n)$-invariant, we have
\begin{equation}
    |P_n\cap (-f_q+P_n)|\geq (1-1/n)|P_n| \label{main13}
\end{equation}
 for every $q\in J_n$. On the other hand, if $q\in s_i^{-1}I_n$, then we have
\begin{equation*}
    a_1^{\epsilon_1k_1}\cdots a_i^{\epsilon_ik_i} a_ia_{i+1}^{\epsilon_{i+1}k_{i+1}}\cdots a_d^{\epsilon_dk_d}=\widehat{s_iq}\in \widehat{I_n}
\end{equation*}
and thus
\begin{align*}
    a_i\widehat{q} (-f_q+P_n) =a_1^{\epsilon_1k_1}\cdots a_i^{\epsilon_ik_i}a_ia_{i+1}^{\epsilon_{i+1}k_{i+1}}\cdots a_d^{\epsilon_dk_d}\cdot (f_q-f_q+P_n)
    \subset \widehat{I_n}P_n
\end{align*}
 by equation (\ref{main15}). It follows that 
 \begin{equation*}
     \bigsqcup_{q\in I_n\cap s_i^{-1}I_n}a_i\widehat{q}(P_n\cap (-f_q+P_n))\subset F_n\cap a_iF_n.
 \end{equation*}
Then by inequality (\ref{main13}), we have
\begin{equation*}
    |F_n\cap a_iF_n|\geq \sum_{q\in J_n\cap s_i^{-1}I_n} |P_n\cap (-f_q+P_n)|\geq |J_n\cap s_i^{-1}I_n| \cdot (1-1/n)|P_n|.
\end{equation*}
Hence we have
\begin{equation*}
    \frac{|F_n\cap a_iF_n|}{|F_n|}\geq \frac{|J_n\cap s_i^{-1}I_n|}{|I_n|}\cdot (1-1/n)\rightarrow 1 \quad (n\rightarrow \infty).
\end{equation*}
\end{proof}

The following lemma is used in the next subsection to show that $F_n$ is a finite-to-one transversal for some subgroups in $G$. The proof is inspired by \cite[the proof of Lemma 10.9]{LL22}.

\begin{lem}\label{main3}
If $K\leq N$ is a finite index subgroup, then $P_n$ is a finite-to-one transversal for $K\cap M_n$ in $M_n$ for every $n$ sufficiently large.
\end{lem}

\begin{proof}
Let $K\leq N$ be a finite index subgroup and set $k=[N:K]$. Since $L_n\rightarrow N$ in Sub$(N)$, there exists $n_0\in\mathbb{N}$ such that $[L_n:L_n\cap K]=k$ for every $n\geq n_0$. Let $n\geq$ max$\{k,n_0\}$. Then 
\begin{equation*}
   kL_n= \left\{\sum_{y\in Y_n}k_y y\in L_n\;\middle|\;k_y\in k\mathbb{Z}\text{ for  every }y\in Y_n\right\}\leq L_n\cap K
\end{equation*}
holds since all $k$-th power is trivial in $L_n/(L_n\cap K)$. Recall that $Y_n$ is a basis of the free abelian group $L_n$ and 
\begin{equation*}
    P_n=\left\{\sum_{y\in Y_n}k_y y\in L_n\;\middle|\;k_y\in [0,l_nn!-1]\text{ for every }y\in Y_n\right\}.
\end{equation*}
Since $n\geq k$, the integer $k$ divides $l_nn!$ and thus $P_n$ is a finite-to-one transversal for $kL_n$ in $L_n$. Hence it is a finite-to-one transversal for $L_n\cap K$ in $L_n$. Now we have $L_n\leq M_n$ and $k=[L_n:L_n\cap K]\leq [M_n:M_n\cap K]\leq [N:K]=k$. Thus $[L_n:L_n\cap K]=[M_n:M_n\cap K]$ holds. It follows that every transversal for $L_n\cap K$ in $L_n$ is also a transversal for $M_n\cap K$ in $M_n$. Hence $P_n$ is a finite-to-one transversal for $M_n\cap K$ in $M_n$.
\end{proof}

\subsection{Construction of the finite index subgroups}\label{mains2}

Recall Notation \ref{pre12}. In this subsection, we prove that the F{\o}lner sequence $(F_n)$ constructed in Lemma \ref{main16} satisfies the following:\\
($\#_R$) If $H\leq G$ is a subgroup with $Q_H=R$ and $[N:\textup{N}_N(H)]<\infty$, then there exists a sequence of finite index subgroups $K_n\leq G$ such that the sequence of pairs $(K_n,F_n)$ is a Weiss approximation of $H$.\par
This part is essentially the same as \cite[section 10]{LL22}, and the lemmas below are proved quite analogously. However the setting is changed at some points, so we give their proofs for the reader's convenience.\par

Let $H\leq G$ be a subgroup with $Q_H=R$ and $[N:$ N$_N(H)]<\infty$. We will construct a sequence of finite index subgroups $K_n\leq G$ with $K_n=[Q_n,N_n,\alpha_n]$. The subgroups $Q_n$ are already defined. Now we define $N_n$.

\begin{lem}[{\cite[Lemma 10.4]{LL22}}]\label{main18}
For every $n\in\mathbb{N}$, there exists a finite index subgroup $N_n\leq N$ such that $O_n\leq N_n,\ N_H\cap M_n\leq N_n,\ N_H\cap T_n=N_n\cap T_n$ and $N_n$ is $Q_n$-invariant.
\end{lem}

\begin{proof}
Let $\Gamma_n:=Q\ltimes N/O_n$ be the semidirect product induced from the $\mathbb{Z}[Q]$-structure on $N/O_n$. It is finitely generated since $N/O_n$ is a finitely generated $\mathbb{Z}[Q]$-module. The subgroup $\pi_n(N_H\cap M_n)\leq N/O_n$ is $Q_n(=R\times V_n)$-invariant. Indeed $N_H$ is $R$-invariant since $Q_H=R$, $M_n$ is $R$-invariant by by Lemma \ref{main2} (ii), and $V_n$ acts on $N/O_n$ trivially by Lemma \ref{main2} (i). Hence 
\begin{equation*}
    \widetilde{H_n}:=Q_n\cdot\pi_n(N_H\cap M_n)\subset\Gamma_n
\end{equation*}
is a subgroup. Since $Q_n$ is finite index in $Q$, the subgroup $Q_n\cdot N/O_n\leq\Gamma_n$ is finitely generated. Then by Proposition \ref{pre3} (take $G,N,H,T$ of Proposition \ref{pre3} to be $\Gamma_n,N/O_n,\widetilde{H_n},\pi_n(T_n)$), there exists a finite index subgroup $\widetilde{N_n}\leq N/O_n$ such that $\pi_n(N_H\cap M_n)\cap \pi_n(T_n)=\widetilde{N_n}\cap\pi_n(T_n),\ \pi_n(N_H\cap M_n)\leq \widetilde{N_n}$ and $\widetilde{N_n}$ is $Q_n$-invariant. Now we set $N_n=\pi_n^{-1}(\widetilde{N_n})$. Clearly $N_H\cap M_n\leq N_n$ and $N_n$ is $Q_n$-invariant. Also ker$\:\pi_n=O_n\leq N_n$ holds. Finally since $T_n\subset M_n$ and $\pi_n|_{M_n}$ is injective by Lemma \ref{main2} (iii), $N_H\cap T_n=N_n\cap T_n$ holds.
\end{proof}

Next we define $\alpha_n$ and $K_n$.

\begin{lem}[{\cite[Lemma 10.6]{LL22}}]\label{main17}
There exist $n_0\in\mathbb{N}$ and finite index subgroups $K_n\leq G$ with $[K_n]=[Q_n,N_n,\alpha_n]$ defined for $n\geq n_0$ such that the sequence $\alpha_n|_R$ is consistent with $\alpha_H$.
\end{lem}

\begin{proof}
Since $N_H\cap T_n=N_n\cap T_n$ and the sequence $T_n$ is increasing to $N$, we have $N_n\rightarrow N_H$ in Sub$(N)$. Then by Proposition \ref{pre4}, there exist homomorphisms $\alpha_n':R\rightarrow$ N$_G(N_n)/N_n$ for all $n$ sufficiently large such that the sequence $\alpha_n'$ is consistent with $\alpha_H$. Then for each $n$, we extend $\alpha_n'$ to $\alpha_n:Q_n\rightarrow$ N$_G(N_n)/N_n$ by $\alpha_n(s_i^{2m_n})=a_i^{2m_n}N_n$ for $c+1\leq i\leq d$. This is well-defined. Indeed since $Q_n=R\times V_n$ and $V_n$ is a free abelian group generated by $\{s_i^{2m_n}\}_{i=c+1}^d$, it suffices to show $[g,a_i^{2m_n}]\in N_n$ for every 
$g\in G$ and $c+1\leq i\leq d$. By the definition of $O_n$, $[g,a_i^{2m_n}]\in O_n\leq N_n$ holds, and it is done. Then $\alpha_n$ satisfies $\overline{\alpha_n(q)}=q$ for every $q\in Q_n$.
There exists a unique $K_n\leq G$ such that $[K_n]=[Q_n,N_n,\alpha_n]$ by Proposition \ref{pre1}.
\end{proof}

Now we show that the sequence $(K_n,F_n)$ is a Weiss approximation of $H$.

\begin{lem}\label{main9}
We have $F_n*H - F_n*K_n \rightarrow 0$ in \textup{C(Sub}$(G))^*$.
\end{lem}

\begin{proof}
It suffices to verify conditions (i)-(v) of Theorem \ref{pre5}. Recall that $Q_H=R$. Conditions (i) follows from Lemma \ref{main5} (i). Condition (ii) follows from Lemma \ref{main2} (ii). Condition (iii) follows from Lemma \ref{main17}. Condition (iv) follows from Lemma \ref{main18}. Finally condition (v) follows from Lemma \ref{main12}.
\end{proof}

\begin{lem}[{\cite[Lemma 10.8]{LL22}}]\label{main6}
For every $n\geq n_0$, $\textup{N}_N(H)\cap M_n\leq \textup{N}_N(K_n)$ holds.
\end{lem}

\begin{proof}
Let $f\in$ N$_N(H)\cap M_n$ and $g\in K_n$. Set $q=\overline{g}\in Q_n$. Then $g^f=g\cdot[g,f]$ and $[g,f]=f^{1-q}\in N$. It suffices to show $f^{1-q}\in N_n$. Write $q=rv$ with $r\in R$ and $v\in V_n$. We have
\begin{equation*}
     f^{1-q}=f^{1-r}+f^{r(1-v)}.
\end{equation*}
Since $Q_H=R$, we can take $h\in H$ such that $\overline{h}=r$. Then $f^{1-r}=[h,f]=h^{-1}\cdot f^{-1}hf\in H$ since $f\in$ N$_N(H)$. Also we have $f^{1-r}\in M_n$ since $f\in M_n$ and $M_n$ is $R$-invariant. It follows that $f^{1-r}\in N_H\cap M_n\leq N_n$. On the other hand, $f^{r(1-v)}\in O_n\leq N_n$ by Lemma \ref{main2} (i). Hence $f^{1-q}\in N_n$ holds.
\end{proof}

\begin{lem}[{\cite[Lemma 10.9]{LL22}}]\label{main7}
For every $n$ sufficiently large, $F_n$ is a finite-to-one transversal for $\textup{N}_G(K_n)$ in $G$.
\end{lem}

\begin{proof}
Since $[N:$ N$_N(H)]<\infty$, the subset $P_n$ is a finite-to-one transversal for N$_N(H)\cap M_n$ in $M_n$ for every $n$ sufficiently large by Lemma \ref{main3}. Fix such $n\geq n_0$. Then $P_n$ is also a finite-to-one transversal for N$_N(K_n)\cap M_n$ in $M_n$ by the previous lemma. By Lemma \ref{main14}, $P_n$ is a finite-to-one transversal for $\pi_n'^{-1}(\text{N}_N(K_n)\cap M_n)$ in $N$. Since ker$\:\pi_n'=O_n\leq N_n\leq\text{N}_N(K_n)$ and $\pi_n'|_{M_n}=$ id$_{M_n}$, we have $\pi_n'^{-1}(\text{N}_N(K_n)\cap M_n)\leq$ N$_N(K_n)$. Hence $P_n$ is a finite-to-one transversal for N$_N(K_n)$ in $N$. On the other hand, $I_n$ is a finite-to-one transversal for $Q_{\text{N}_G(K_n)}$ in $Q$ by Lemma \ref{main5} (ii) and $Q_n\leq Q_{\text{N}_G(K_n)}$. Hence $F_n=\widehat{I_n}P_n$ is a finite-to-one transversal for N$_G(K_n)$ in $G$ by Proposition \ref{pre2}.
\end{proof}

By Lemmas \ref{main9} and \ref{main7}, the sequence $(K_n,F_n)$ is a Weiss approximation of $H$ by ignoring small $n$. Hence property ($\#_R$) holds, and Theorem \ref{intr1} is proved.

\section{Limits of P-stable groups}\label{unis}

\begin{lem}\label{uni1}
Let $G$ be a countable group and $G_k\leq G$ a sequence of subgroups such that $G_k\rightarrow G$ in \textup{Sub}$(G)$. Assume that for every $k$, there exists a homomorphism $p_k:G\rightarrow G_k$ such that $p_k|_{G_k}=\textup{id}_{G_k}$. If $G_k$ is P-stable for every $k$, then $G$ is P-stable.
\end{lem}

\begin{proof}
Let $(\phi_n:G\rightarrow$ Sym$(n))_n$ be an almost homomorphism. Let $F\subset G$ be a finite subset. It suffices to show that there exists a sequence of homomorphisms $(\psi_n:G\rightarrow$ Sym$(n))_n$ such that $d_\text{H}(\phi_n(g),\psi_n(g))\rightarrow 0\ (n\rightarrow\infty)$ for every $g\in F$. Fix $k$ such that $F\subset G_k$. Since $(\phi_n|_{G_k}:G_k\rightarrow \text{Sym}(n))_n$ is an almost homomorphism, there exists a sequence of homomorphisms $(\psi_n':G_k\rightarrow$ Sym$(n))_n$ such that $d_\text{H}(\phi_n(g),\psi_n'(g))\rightarrow 0\ (n\rightarrow\infty)$ for every $g\in F$. Then the sequence of homomorphisms $\psi_n:=\psi_n'\circ p_k$ satisfies the required condition.
\end{proof}

\begin{proof}[Proof of Corollary \ref{intr4}]
Let $F_\infty=\langle s_1,s_2,...\rangle$ and $F_k=\langle s_1,...,s_k\rangle\leq F_\infty$ for every $k\in\mathbb{N}$. Let $i_k:F_k/{F_k''}\rightarrow F_\infty/{F_\infty''}$ be the homomorphism induced from the inclusion $F_k\rightarrow F_\infty$. Also let $p_k:F_\infty/{F_\infty''}\rightarrow F_k/{F_k''}$ be the homomorphism induced from the natural projection $F_\infty\rightarrow F_k$. Then we have $p_k\circ i_k=$ id$_{F_k/{F_k''}}$. Since $F_k/{F_k''}$ is P-stable for every $k$ by Theorem \ref{intr1}, so is $F_\infty/{F_\infty''}$ by Lemma \ref{uni1}.
\end{proof}

\footnotesize
\textsc{Graduate School of Mathematical Sciences, the University of Tokyo, 3-8-1 Komaba, Tokyo 153-8914, Japan.}\par
\textit{E-mail address}: \texttt{ishikura8000@g.ecc.u-tokyo.ac.jp}


\begin{thebibliography}{20}
\bibitem[AGV14]{AGV14} M. Abért, Y. Glasner and B. Virág, Kesten’s theorem for invariant random subgroups, \textit{Duke Math. J.} \textbf{163} (2014), 465--488.
\bibitem[AP15]{AP15} G. Arzhantseva and L. Păunescu, Almost commuting permutations are near commuting permutations, \textit{J. Funct. Anal.} \textbf{269} (2015), 745--757.
\bibitem[B65]{B65} S. Bachmuth, Automorphisms of free metabelian groups, \textit{Trans. Amer. Math. Soc.} \textbf{118} (1965), 93--104.
\bibitem[BLT19]{BLT19} O. Becker, A. Lubotzky and A. Thom, Stability and invariant random subgroups, \textit{Duke Math. J.} \textbf{168} (2019), 2207--2234.
\bibitem[G18]{G18} T. Gelander, A view on invariant random subgroups and lattices, \textit{Proceedings of the International Congress of Mathematicians--Rio de Janeiro 2018.} Vol. II. Invited lectures, 1321--1344, World Sci. Publ., Hackensack, NJ, 2018.
\bibitem[G10]{G10} L. Glebsky, Almost commuting matrices with respect to normalized Hilbert-Schmidt norm, \textit{preprint,} arXiv:1002.3082.
\bibitem[GR09]{GR09} L. Glebsky and L. M. Rivera, Almost solutions of equations in permutations, \textit{Taiwanese J. Math.} \textbf{13} (2009), 493--500.
\bibitem[LL19]{LL19} A. Levit and A. Lubotzky, Uncountably many permutation stable groups, \textit{preprint,} arXiv:1910.11722.
\bibitem[LL22]{LL22} A. Levit and A. Lubotzky, Infinitely presented permutation stable groups and invariant random subgroups of metabelian groups, \textit{Ergodic Theory Dynam. Systems} \textbf{42} (2022), 2028--2063.
\bibitem[LV22]{LV22} A. Levit and I. Vigdorovich, Characters of solvable groups, Hilbert-Schmidt stability and dense periodic measures, \textit{preprint}, arXiv:2206.02268.
\bibitem[LS01]{LS01} R. C. Lyndon and P. E. Schupp, \textit{Combinatorial group theory}, Reprint of the 1977 edition, Classics in Mathematics, Springer-Verlag, Berlin, 2001.
\bibitem[V83]{V83} D. Voiculescu, Asymptotically commuting finite rank unitary operators without commuting approximants, \textit{Acta Sci. Math. (Szeged)} \textbf{45} (1983), 429--431.
\bibitem[Z19]{Z19} T. Zheng, On rigid stabilizers and invariant random subgroups of groups of homeomorphisms, \textit{preprint,} arXiv:1901.04428.
\end{thebibliography}
\end{document}